\documentclass[11pt,reqno]{amsart}

\usepackage[american]{babel}
\usepackage{comment}
\usepackage{amsmath,mathrsfs}
\usepackage{dsfont,mathtools,amssymb}
\usepackage{color}
\mathtoolsset{showonlyrefs,showmanualtags}
\usepackage[colorlinks=true, linkcolor=black, citecolor=black, urlcolor=black, hidelinks]{hyperref}

\newtheorem{theorem}{Theorem}[section]
\newtheorem{proposition}[theorem]{Proposition}
\newtheorem{lemma}[theorem]{Lemma}

\newtheorem{remark}[theorem]{Remark}
\newtheorem{definition}[theorem]{Definition}

\numberwithin{equation}{section}
\numberwithin{figure}{section}

\newcommand{\T}{\mathbb{T}}
\newcommand{\Ex}{\mathds{E}}
\newcommand{\Prob}{\mathds{P}}
\renewcommand{\P}{\mathbf{P}}
\newcommand{\E}{\mathbf{E}}

\newcommand{\eps}{\varepsilon}

\newcommand{\R}{\mathbb{R}}
\newcommand{\N}{\mathbb{N}}

\newcommand{\be}{\begin{align}}

	\newcommand{\cG}{\ensuremath{\mathcal G}}

	\newcommand{\cW}{\ensuremath{\mathcal W}}

	\def\({\left(}
	\def\){\right)}

	\newcommand{\rw}{{\textsc{rw}}}

	\allowdisplaybreaks
	
	\title[$L^2$-cutoff for avg on random regular graphs]{$L^2$-cutoff for the averaging process \\
		on random regular graphs}
	
	
	\author{Pietro Caputo, Matteo Quattropani,  and Federico Sau}
	\address{Pietro Caputo\\ Universit\`{a} degli Studi Roma Tre}
	\email{pietro.caputo@uniroma3.it}
	\address{Matteo Quattropani\\  Universit\`{a} degli Studi Roma Tre}
	\email{matteo.quattropani@uniroma3.it}
	\address{Federico Sau\\ Università degli Studi di Milano}
	\email{federico.sau@unimi.it}

\begin{document}
		\maketitle
		\begin{center}
			\emph{Dedicated to Claudio Landim on the occasion of his 60th birthday.}
		\end{center}
		\begin{abstract}
We study the mixing time of the averaging process on  a large random $d$-regular graph, $d\ge 3$, and prove an $L^2$-cutoff with an explicit cutoff time. Somewhat surprisingly, we uncover a phase transition at the finite, fixed degree $d=10$: for small degrees, i.e., $d\le 10$, the averaging process mixes as fast as the corresponding random walk on the same graph, whereas for $d> 10$ its $L^2$-mixing is governed by a different, slower mechanism.
Our proof relies on a detailed asymptotic analysis of an auxiliary biased birth-and-death chain with a slow bond. We also briefly discuss an analogous phase transition for the $L^1$-mixing.
		\end{abstract}
%\tableofcontents		
\section{Introduction, model, and main result}
	
	The quantitative study of convergence to equilibrium for stochastic processes is a classical theme in probability theory.
	A central object in this direction is the {mixing profile} of a Markov chain and, in particular, the possible
	presence of the \emph{cutoff phenomenon}: an abrupt transition from non-equilibrium to equilibrium over a time window
	negligible compared to the mixing time.
	Since its identification in the early 1980s through the pioneering work of Aldous and Diaconis \cite{aldous1986shuffling}, explaining the occurrence of cutoff has become a
	major problem in the modern theory of mixing times.
	Despite a rapidly growing collection of examples and techniques, a general mechanism explaining when cutoff should
	(or should not) occur is still missing; see, e.g., \cite{levin2017markov,salez_cutoff_2024,salez2025modern} for more background.
	
	Among the first and most influential  examples exhibiting cutoff  are \emph{random walks on sparse random graphs}. Here, the interplay
	between geometric expansion and randomness gives rise to sharp mixing transitions (see, e.g., \cite{lubetzky2010cutoff,lubetzky_peres_ramanujan_2016,ben-hamou_salez_2017,berestycki_random_2018,bordenave_random_2018,bordenave_cutoff2019} and references
	therein).
	These models provide a natural testing ground for cutoff questions, as they often allow one to relate mixing to
	spectral information, local tree-like structure, and concentration phenomena.
	
Comparatively, mixing-time results for continuous-state Markov processes are more sparse, and sharp cutoff
statements are known only for selected model classes; see, e.g.,
\cite{smith_analysis_2013,smith_gibbs_2014,barrera_thermalisation2020,caputo_mixing_2019,caputo_labbe_lacoin_spectral_2022,salez2025cutoff,caputo2025universal}.
 An example in this class is the \emph{averaging process}, which is particularly simple to describe, yet it exhibits a rather interesting behavior. Since the seminal work of \cite{aldous_lecture_2012}, the literature on the averaging process has grown rapidly; see, e.g., \cite{chatterjee2020phase,quattropani2021mixing,caputo_quattropani_sau_cutoff_2023,movassagh_repeated2022,sau_concentration_2024,sau_tiny_2024,eide_concentration_2024,gantert_vilkas_averaging_2025,elboim_peres_peretz_edge_2025,campos_hydrodynamic_2025,caputo2024repeated,caputo2025universal}.

  Let us now describe the process.
	Given a graph $G=(V,E)$, fix an initial nonnegative mass
	configuration $\eta=(\eta(x))_{x\in V}$ with $\sum_{x\in V}\eta(x)=1$, and equip each edge of the graph with an independent rate-$1$ Poisson clock.
	Whenever the clock on an edge $xy\in E$ rings at time $t>0$, the masses at $x$ and $y$ are replaced by their average: if $\eta_s$ denotes the mass configuration at time $s\ge 0$, and $\eta_s(z)\in [0,1]$ its mass at $z\in V$, 
	\begin{equation}
	\eta_t(x)=\eta_t(y)=\frac{\eta_{t^-}(x)+\eta_{t^-}(y)}{2}\ ,\qquad \eta_t(z)=\eta_{t^-}(z)\ \text{for all}\ z\neq x, y\ .
	\end{equation}
	
	If $G$ is connected, then, for every initial configuration, the process converges, as $t\to \infty$,	 to the flat one
	$\pi\equiv(1/n,\dots,1/n)$.
	It is worth emphasizing, however, the notion of “distance to equilibrium” under which this convergence is measured.
	Since the steady state is a Dirac delta at $\pi$ and $\eta_t$ is, generally, genuinely random at all times $t>0$, convergence in total variation is too strong.
	For this reason, the literature  has focused instead on Wasserstein-type distances, and, in particular, on
	\begin{equation}\label{eq:wasserstein}
		 \E_{\eta}\!\left[\left\|\frac{\eta_t}{\pi}-1\right\|_{p}^{\,p}\right],
		\qquad p\ge 1\ ,\ t\ge 0\ .
	\end{equation}
	Here, $\|\cdot\|_{p}$ denotes the usual $L^p(\pi)$-norms of functions on $V$, while $\mathbf P_\eta$ and $\mathbf E_\eta$ denote, respectively,  the law and expectation associated to $(\eta_t)_{t\ge 0}$ when $\eta_0=\eta$. It turns out that the quantities in \eqref{eq:wasserstein} are non-increasing in $t$ and are maximized by initial conditions
	that concentrate all the mass at a single vertex. For further details and properties, we refer the interested reader to \cite[Section 2]{caputo_quattropani_sau_cutoff_2023}.

When analyzing the averaging process' mixing, it is natural to compare $(\eta_t)_{t\ge 0}$ with the \textquotedblleft lazy\textquotedblright\ random walk $(X_t)_{t\ge0}$ on the same graph $G$ (hereafter, also referred to as ${\rm RW}(G)$), that is, the process on $V$ which jumps from $x$ to $y$ at rate $1/2$, provided that $xy\in E$. 	
	Indeed, letting $(\pi^\eta_t)_{t\ge0}$ denote the transition probabilities of $(X_t)_{t\ge 0}$ started with distribution $\eta$, one has (see, e.g., \cite{quattropani2021mixing})
	\begin{equation}\label{eq:AVG-RW}
		\pi^\eta_t(x)=\E_{\eta}[\eta_t(x)]\ ,\qquad x\in V\, ,\ t\ge 0\ .
	\end{equation}
As a consequence of this identity and Jensen inequality,  one obtains the following general comparison between the mixing behavior of the averaging process and the random walk on the same graph:
	\begin{equation}\label{eq:comparison}
	 \E_{\eta}\!\left[\left\|\frac{\eta_t}{\pi}-1\right\|_{p}^{p}\right]\ge
		\left\|\frac{\pi_t^\eta}{\pi}-1 \right\|_{p}^{\,p}\,,\qquad p\ge 1\ ,\ t\ge 0\ .
	\end{equation}
	This inequality may be interpreted as saying that, for a fixed initial condition and on the same graph, $(\eta_t)_{t\ge 0}$ approaches equilibrium
	no faster than $(X_t)_{t\ge 0}$ on the same graph.

	The aim of this work is to investigate the mixing behavior of the averaging process on a (large) random $d$-regular graph
	(with fixed degree $d\ge3$) in the $L^2$-metric, when starting with all the mass in one  vertex.	
	Our main result highlights two main facts.
	First, for every $d\ge3$, with high probability over the choice of the graph and of the starting vertex,  the
	averaging process on random regular graphs exhibits an \emph{$L^2$-cutoff}.
		Second, and maybe more unexpectedly, we identify a \textit{finite} degree threshold in the comparison with the random walk: in the large-graph limit and for $p=2$, the inequality in
	\eqref{eq:comparison} is, asymptotically, in fact an equality if $d\le 10$, while it becomes strict as soon as
	$d>10$.
	In particular, beyond degree $d=10$ the averaging process exhibits an intrinsic slowdown relative to the random walk.
	
		\subsection{Main result}		
		For all integers $n\ge 2$ and $d\ge 3$, let $\cG(n,d)$
		be the set of all labeled simple undirected graphs with vertex set $V$, $|V|=n$,
		and all degrees equal to $d$. Let $G$ be a random simple $d$-regular graph of
		size $n$, i.e., an element of $\cG(n,d)$ sampled uniformly at random.
		Equivalently, $G$ can be generated via the configuration model conditioned on
		the event that the resulting multigraph is simple. This graph $G=(V,E)$
		constitutes the underlying random environment, with $\Prob$ and
		$\Ex$ denoting the corresponding law and expectation.
		
		For a given (quenched) realization of $G$, we let the averaging process run on it, and set $\eta_0=\delta_1$, that is,  all the mass is initially concentrated at a generic vertex, say,  with label $1$. 
		This is our main result, whose proof is postponed to Section \ref{sec:proof-main} below.
		\begin{theorem}[$L^2$-cutoff]\label{th:L2}
			Let, for every integer $d\ge 3$,
			\begin{equation}\label{eq:gamma2}
				\gamma=\gamma_d:= \begin{dcases}
					\varrho &\text{if}\ d\le 10\\
					\sigma &\text{if}\ d \ge 10\ ,
				\end{dcases}
			\end{equation}
			where
			\begin{equation}\label{eq:gamma}
			\varrho=\varrho_d:=\frac{2\sqrt{d-1}}{d}\ ,\qquad \sigma=\sigma_d:= \frac{\left(3\sqrt{d-1}-\sqrt{d-9}\right)\sqrt{d-1}}{4d}\ .
			\end{equation}
	Further, define (note that $\gamma \in (0,1)$, for all $d\ge 3$; see \eqref{eq:gamma-properties} below)
			\begin{equation}\label{eq:def-Ta}
				T(a):= \frac{1}{d\left(1-\gamma\right)}\log n+ a\log \log n\ ,\qquad a \in \R\ . 
			\end{equation}
		Then, 
		for every $\varepsilon> 0$, we have
			\begin{align}\label{eq:theorem-lb}	
				\lim_{a\to -\infty}\liminf_{n\to \infty} \Prob\bigg(\E_{\delta_1}\!
				\left[\left\|\frac{\eta_{T(a)}}{\pi}-1\right\|_2^2\right]>\eps\bigg)&=1\ ,
		\\ 
		\label{eq:theorem-ub}
				\lim_{a\to +\infty}\limsup_{n\to \infty} \Prob\bigg(\E_{\delta_1}\!\left[\left\|\frac{\eta_{T(a)}}{\pi}-1\right\|_2^2\right]<\eps\bigg)&=1\ .
		\end{align}

		\end{theorem}
		In words, this theorem establishes that the averaging process on large  random $d$-regular graphs exhibits, with high probability (w.h.p.) with respect to the law of the underlying random graph and from a typical starting vertex, an $L^2$-cutoff at time $\frac1{d\left(1-\gamma\right)}\log n$, with a cutoff window of size at most $\log \log n$.
	We collect below some further insights around our main result, and some related open questions.
Throughout, all asymptotic statements are taken in the limit $n\to \infty$ (unless stated otherwise). We use standard asymptotic notation: for positive functions $f$ and  $g$, we write  $f=O(g)$ (resp.\ $o(g)$) if $\limsup (f/g)<\infty$ (resp.\ $\lim (f/g)=0$). We write $f\asymp g$ if $f=O(g)$ and $g=O(f)$, whereas $f\sim g$ if $\lim (f/g)=1$. 
	\subsection{Further remarks and open problems}

\subsubsection{Comparison with the random walk} 
	\begin{figure}[t]
	\centering
	\includegraphics[width=6cm]{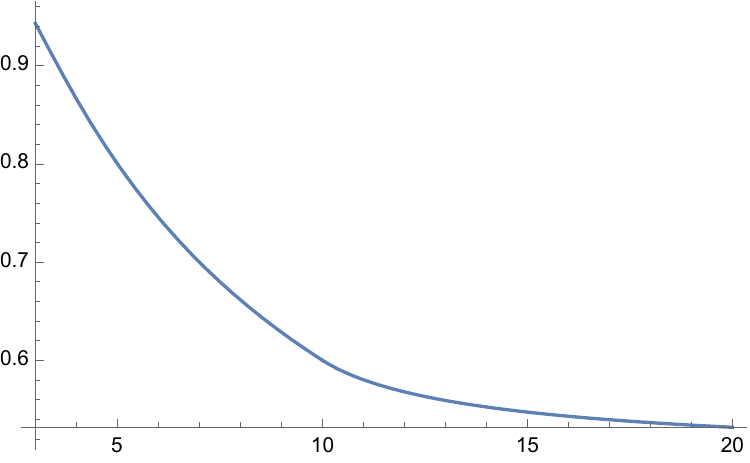}
	\qquad
	\includegraphics[width=6cm]{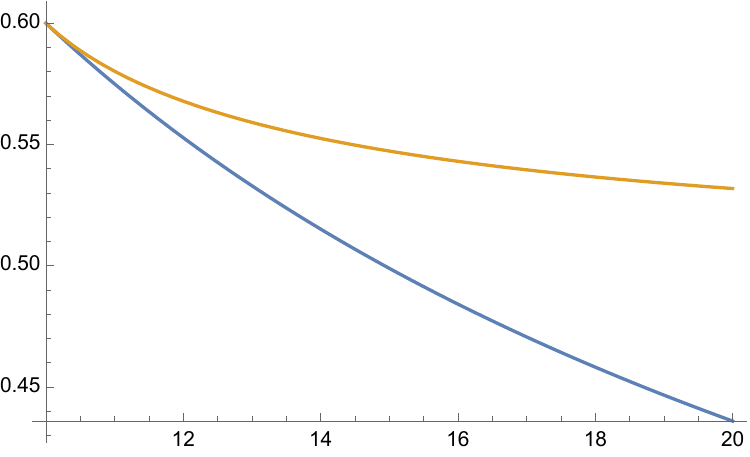}
	\caption{\textit{Left}: Plot of the function $d\mapsto \gamma_d$, for $d\ge 3$. \textit{Right}: Plots of the functions $d\mapsto \varrho_d$ (blue) and   $d\mapsto \sigma_d$ (orange), for $d\ge 10$. }
	\label{fig:gamma}
\end{figure}

	The function  $\gamma=\gamma_d$ given in \eqref{eq:gamma} satisfies the following three simple properties (see also Figure \ref{fig:gamma}):
	\begin{equation}\label{eq:gamma-properties}
		\gamma\ \text{is strictly decreasing in}\ d\ge 3\ ,\qquad \gamma=\frac{2\sqrt 2}3\ \text{for}\ d=3\ ,\qquad \gamma\xrightarrow{d\to \infty}\frac12\ .
	\end{equation}
Moreover, as shown in Lemma \ref{lemma:rd} below (see also Figure \ref{fig:gamma}), 
	\begin{equation}
		\sigma >\varrho\ ,\qquad d>10\ .
	\end{equation}
	This is the quantitative mark that the averaging process $L^2$-mixing time is strictly larger than its ${\rm RW}$ counterpart when $d>10$; for $d\le 10$, since $\gamma =\varrho$, the $L^2$-mixing of the two processes coincide, at first order. Indeed, by adapting \cite[Proposition 6]{lubetzky_peres_ramanujan_2016} to our continuous-time (lazy) setup, their result reads as follows:
	for every $\varepsilon> 0$ and  w.h.p., 
	\begin{equation}
		t_{\rm mix,2}^\rw(\eps)\coloneqq \inf\bigg\{t>0: \bigg\|\frac{\pi_t^{\delta_1}}{\pi}-1\bigg\|_2^2<\eps\bigg\}\sim \frac1{d(1-\varrho)}\log n\ ,
	\end{equation}
	while, by Theorem \ref{th:L2}, one has for the averaging process
	\begin{equation}
		t_{\rm mix,2}(\eps)\coloneqq \inf\bigg\{t>0: \E_{\delta_1}\!\left[\left\|\frac{\eta_t}{\pi}-1\right\|_2^2\right]<\eps\bigg\}\sim \frac1{d(1-\gamma)}\log n\ ,
	\end{equation}
	and the two right-hand sides above coincide only for $d\le 10$. 
	
	See Remark \ref{rem:transition} below for a further motivation lying behind this transition.

\subsubsection{Large-degree regime}
It is well known that, when the degree 
$d$ of a random 
$d$-regular graph grows sufficiently quickly as 
$n\to \infty$, many of its properties align with those of a mean-field model.
Mixing, both in $L^1$ and $L^2$, of the averaging process on the complete graph with $n$ vertices has been studied in \cite{chatterjee2020phase}, and cutoff has been proved for both metrics.
The $L^2$-cutoff time  is given (in continuous-time) by $\frac{2}{n}\log n$, with a cutoff window of size $\frac1n$. By formally taking $d\sim n$ in Theorem \ref{th:L2}, since $\gamma\to\frac{1}{2}$ as $d\to \infty$ (cf.\ \eqref{eq:gamma-properties}), one recovers the leading-order term of the mixing time on the complete graph. 

Actually, after an inspection of the proof of Theorem \ref{th:L2}, one may note that our strategy for the lower bound---ultimately based on \eqref{eq:l2-norm}, \eqref{eq:lb-CRW(G)-Td}, \eqref{eq:CRW-infinite-tree-lb} below---is robust enough to rigorously justify a joint limit $d, 	n\to \infty$ in \eqref{eq:theorem-lb}, and predicts an $L^2$-mixing time asymptotically equal to $\frac2d\log n$. Our arguments for the upper bound (in particular, \eqref{eq:markov-l2}) are too weak, though. Nevertheless, one may use a more direct and elementary approach via Aldous-Lanoue's inequality for the averaging process \cite[Proposition 2]{aldous_lecture_2012}, which holds  on any graph $G$: for all mass configurations $\eta$, 
\begin{equation}\label{eq:AL}
	\E_\eta\!\left[\left\|\frac{\eta_t}{\pi}-1\right\|_2^2\right]\le e^{-{\rm gap}\,t} \left\|\frac{\eta}\pi-1\right\|_2^2\ ,\qquad t \ge 0\ ,
\end{equation} 
where  ${\rm gap}={\rm gap}(G)\ge 0$ denotes the spectral gap of ${\rm RW}(G)$.
 By exploiting  that, for random $d$-regular graphs on $n$ vertices with $d, n\to \infty$, one has, w.h.p.,  ${\rm gap}\sim \frac{d}2$ (see, e.g., \cite{bauerschmidt_edge_2020,he_spectral_2024}), one recovers from \eqref{eq:AL} $\frac{2}d\log n$ as a matching upper bound.

For further details on this regime and the size of the cutoff window, we refer to Remark  \ref{rem:window}.

\subsubsection{$L^p$-cutoff and phase transitions}We focus on 
$L^2$-mixing; extending the analysis to 
$L^p$ (especially 
$p=1$) is a natural direction. While we do not pursue a detailed analysis here, relatively soft arguments (see, e.g., \cite[Section 2]{caputo_quattropani_sau_cutoff_2023} for a survey)  already allow us to derive some $L^p$-bounds, $p\in [1,2]$, and suggest an analogous phase transition for $p=1$, as we now briefly explain.

Thanks to the monotonicity in $p\ge 1$ of the metrics, that is,
\begin{align}
	\E_\eta\!\left[\left\|\frac{\eta_t}{\pi}-1\right\|_p^p\right]\le 	\E_\eta\!\left[\left\|\frac{\eta_t}{\pi}-1\right\|_q^q\right]\ ,\qquad 1\le p\le q\ ,\ t\ge 0\ ,
\end{align}
our Theorem \ref{th:L2} already provides $\frac1{d(1-\gamma)}\log n+O(\log \log n)$ as an upper bound for the $L^p$-mixing time, for all $p\le 2$. 

For the lower bound, we can proceed in two ways. One route compares the averaging process
to the random walk via \eqref{eq:comparison}, together with \cite{lubetzky2010cutoff} and
\cite[Proposition~6]{lubetzky_peres_ramanujan_2016}. Another route uses the
relative-entropy lower bound of \cite{movassagh_repeated2022}, see also \cite[Proposition 2.9]{caputo_quattropani_sau_cutoff_2023}.
Heuristically, these two
bounds capture different mechanisms required for the averaging process to be well mixed:
the comparison argument reflects the need for the mass to \emph{explore}, in average, a substantial
portion of the graph, whereas the entropy argument reflects the need for sufficient
\emph{fragmentation}. In either case, we obtain lower bounds of the form $c\log n$, for some $c>0$ depending on $d\ge 3$ and possibly on the accuracy parameter $\eps>0$ in the mixing time $t_{{\rm mix},p}(\eps)$. This yields a pre-cutoff (see, e.g., \cite[Eq.\ (18.5)]{levin2017markov}) for all $p\in[1,2)$.

Further, if specialized to the $L^1$-setting, these two lower bounds read, respectively, as
\begin{align}
	t_{{\rm mix},1}(\eps)&\ge \frac{2}{d-2}\frac1{\log(d-1)}\log n-O(\sqrt{\log n})\ ,\\
	t_{{\rm mix},1}(\eps)&\ge \frac{1-\eps}{d\log 2}\log n-O(\sqrt{\log n})\ ,
\end{align}
w.h.p.\ and for all $\eps \in (0,1)$ small enough.
A simple manipulation shows that, if $d> 7$, the second lower bound obtained via \cite[Theorem 1]{movassagh_repeated2022} improves upon the first one, which we know to coincide with the $L^1$-cutoff time of the random walk on random $d$-regular graphs, $d\ge 3$. In other words, a finite-degree phase transition for the averaging process' $L^1$-mixing may occur at some $3\le d\le 7$. Establishing the occurrence of $L^1$-cutoff and, possibly,  its precise location for the averaging process on this and other random graph models remains an interesting open problem.

\section{Proof of Theorem \ref{th:L2}}
\label{sec:proof-main}
In this section, we prove our main result, and start by introducing some auxiliary processes which will play a central role in our analysis.

\subsection{Coupled random walks and the slow-bond birth-and-death chain}
As already exploited in  \cite{aldous_lecture_2012}, there is a two-walk version of \eqref{eq:AVG-RW}: for all graphs $G=(V,E)$ and initial configurations $\eta$, 
\begin{equation}
	\E_\eta[\eta_t(x)\eta_t(y)]=\P_{\eta\otimes \eta}^{\scriptscriptstyle{\rm CRW}(G)}(X_t=x,Y_t=y)\ ,\qquad x,y \in V\ ,\ t\ge 0\ .
\end{equation}
Here, $\P^{{\scriptscriptstyle{\rm CRW}(G)}}$ stands for  the law of two 	\textit{coupled random walks} on $V\times V$, and  the subscript indicates the initial distribution of the two walks; when $\eta = \delta_x$, $x\in V$, we simply write $\P_{x,x}^{\scriptscriptstyle{\rm CRW}(G)}$.
Their dynamics may be briefly described as follows (for more details, see, e.g., \cite[Section 2.4]{aldous_lecture_2012} or \cite[Section 2.2]{caputo_quattropani_sau_cutoff_2023}): if the rate-1 Poisson clock attached to $xy\in E$ rings, each walker currently at either $x$ or $y$ independently crosses the edge with probability $1/2$. In particular, both $X_t$ and $Y_t$ marginally evolve as a  ${\rm RW}(G)$, and their motion is independent  as long as they sit at graph-distance larger than one. 

By spelling out $\|\frac{\eta_t}{\pi}-1\|_2^2= n\sum_{x\in V}\eta_t(x)^2-1$ and taking expectation, one readily gets 
\begin{equation}
	\label{eq:l2-norm}
	\E_x\!\left[\left\|\frac{\eta_t}{\pi}-1\right\|_{2}^2\right] =
	n\, \P^{\scriptscriptstyle{\rm CRW}(G)}_{x,x}(X_t=Y_t) -1\ ,\qquad x\in V\ ,\  t \ge 0\ 
\end{equation}
The general idea behind the proof of Theorem \ref{th:L2} is to effectively compare the (random) probability on the right-hand side of \eqref{eq:l2-norm} 	with
\begin{align}\label{eq:crw-Td}
	\P^{\scriptscriptstyle{\rm CRW}(\T_d)}_{o,o}(X_t=Y_t)\ ,\qquad t \ge 0\ ,
\end{align}
where $\mathds T_d$ denotes the $d$-regular infinite tree, $o$ is an arbitrary vertex (which we will refer to as the \emph{root}), and ${\rm CRW}(\mathds T_d)$ stands for a system of coupled random walks evolving on $\mathds T_d$. Analogously, $\P_{o,o}^{\scriptscriptstyle{\rm CRW}(\T_d)}$ denotes the law of ${\rm CRW}(\T_d)$ when both walks start from $o$.	

The advantage of working with \eqref{eq:crw-Td} is that it may be recast into a transition probability of a continuous-time birth-and-death (BD) chain, whose transition probabilities can be asymptotically estimated.
\begin{definition}[Slow-bond BD chain] \label{def:BD-chain}Fix $d\ge 3$.
	Let $R_t$ denote the transition kernel of the  continuous-time biased random walk on $\N_0:=\{0,1,2,\ldots\}$ with transition rates $p^R(i)$ (resp.\ $q^R(i)$) of jumping from $i$ to $i+1$ (resp.\ $i-1$) given by
	\begin{equation}\label{eq:rates-slow-bond}
		p^R(i):= \begin{dcases}
			d/2 &\text{if}\ i=0\\
			d-1 &\text{if}\ i\ge 1
		\end{dcases}\ ,\qquad q^R(i):= \begin{dcases}
		0 &\text{if}\ i=0\\
			1/2 &\text{if}\ i=1\\
			1 &\text{if}\ i\ge2\ .
		\end{dcases}
	\end{equation}
	We refer to $R_t$ as the (continuous-time) \emph{slow-bond BD 	chain}.
\end{definition}  
From the definition of ${\rm CRW}(\T_d)$ and $R_t$, it is immediate to verify that, 	for all $t\ge 0$,
\begin{equation}\label{eq:P-CRW-R}
	\P^{{\rm CRW}(\T_d)}_{o,o}(X_t=Y_t) = R_t(0,0)\ .
\end{equation}
This identity will be the first key step for deriving the following sharp asymptotics for times 
$t=T(a)$ of the form  \eqref{eq:def-Ta}. Its proof takes the whole Section \ref{sec:CRW-infinite-tree}.
\begin{proposition}\label{prop:CRW-infinite-tree-CT}
	For all $d\ge 3$, we have
	\begin{equation}\label{eq:CRW-infinite-tree-lb}
		\lim_{a\to -\infty}\liminf_{n\to \infty}n\,\P^{{\rm CRW}(\mathds T_d)}_{o,o}(X_{T(a)}=Y_{T(a)})=\infty\ ,
	\end{equation}
	\begin{equation}\label{eq:CRW-infinite-tree-ub}
		\lim_{a\to \infty}\limsup_{n\to \infty}	n\,\P^{{\rm CRW}(\mathds T_d)}_{o,o}(X_{T(a)}=Y_{T(a)})=0\ .
	\end{equation}
\end{proposition}

Granted the validity of this result and the above relations,  in the next two sections, we turn to the proof of the two claims in Theorem \ref{th:L2}.

	\begin{remark}\label{rem:transition}
		A departure from the ${\rm RW}$ behavior for the $L^2$-mixing time has to be expected for sufficiently large (but finite) $d\ge3$, as the following  crude argument shows. 

		By \eqref{eq:l2-norm}, the $L^2$-mixing is dictated by transition probabilities that two coupled random walks starting from the same site sit together at time $t$. By estimating this probability from below with the probability that the two walks stay together up to time $t> 0$, we get
		\begin{align}
			\P^{{\rm CRW}(G)}_{x,x}(X_t=Y_t)\ge \P^{{\rm CRW}(G)}_{x,x}(X_s=Y_s\,,\,\text{for all}\ s\le t) = e^{-\frac{d}{2}t}\ ,
		\end{align}
		yielding
		$t_{{\rm mix},2}(\eps)\ge \frac2{d}\log n+\frac{a}{d}$, for all $\eps>0$ and  for some $a=a_\eps\in \R$. This lower bound is strictly larger than $\frac{1}{d(1-\varrho)}\log n$ as soon as $d\ge 15$. Nevertheless, on the one hand, this argument does not yet capture the critical value $d=10$ and, on the other, it does not show that, for $d\ge3$ small enough, the inequality in \eqref{eq:comparison} is (asymptotically) an identity.
	\end{remark}

	\subsection{Lower bound}
In view of \eqref{eq:comparison} with $p=2$,   which holds $\Prob$-a.s.,  and  \cite[Lemma 2.2]{lubetzky_peres_ramanujan_2016}, we readily obtain a  lower bound for Theorem \ref{th:L2}, which turns out to be sharp  whenever $3\le d\le 10$.
However,  such an estimate needs to be refined to derive the lower bound for $d> 10$. For this purpose, recall \eqref{eq:l2-norm}.
By extending the arguments in the proof of \cite[Lemma 2.2]{lubetzky_peres_ramanujan_2016}, we prove the following lower bound, which holds for any deterministic $d$-regular graph.
\begin{proposition}\label{prop:lb}
	Let $G=(V,E)$ be any (non-random) $d$-regular graph, $d\ge 3$. Then, 
		\begin{equation}\label{eq:lb-CRW(G)-Td}
			\inf_{x\in V}\P^{\scriptscriptstyle{\rm CRW}(G)}_{x,x}(X_t=Y_t)\ge \P^{\scriptscriptstyle{\rm CRW}(\T_d)}_{o,o}(X_t=Y_t)\ ,\qquad t \ge 0\ ,
		\end{equation}
		where we recall that $\mathds T_d$ denotes the $d$-regular infinite tree and $o$ is an arbitrary vertex.
\end{proposition}
\begin{proof}[Proof of Theorem \ref{th:L2} (lower bound)]
	The claim in \eqref{eq:theorem-lb} readily follows from \eqref{eq:l2-norm}, Proposition \ref{prop:lb},	 and \eqref{eq:CRW-infinite-tree-lb}.	
\end{proof}

The remainder of this section is devoted to the proof of Proposition \ref{prop:lb}.
 
Fix an integer $d\ge 3$, and let $G=(V,E)$ be a  (non-random) $d$-regular graph, with a given $x\in V$ fixed all throughout.  Let us consider the cover tree of $G$ with root $o\in \T_d$ labeled as $x\in V$. This prescribes a unique covering map (i.e., a graph homomorphism) $\varphi:\T_d\to G$ with $\varphi(0)=x$.

We define a Markov process  $(Z_t,W_t)_{t\ge0}$ (shortened as ${\rm CRW}_*(\T_d)$) on $\T_d\times \T_d$
driven by the rate-$1$ clocks on edges $xy\in E$ as follows.
Whenever the clock of $xy$ rings at time $t>0$, the update coincides with the usual ${\rm CRW}(\T_d)$
update along lifts of $xy$, except when
\begin{equation}\label{eq:teleport}
	{\rm dist}_{\T_d}(Z_{t-},W_{t-})\ge2
	\qquad\text{and}\qquad
	\varphi(Z_{t-}),\varphi(W_{t-})\subset \{x,y\}\ ,
\end{equation}
in which we perform an additional ``teleport'' update (specified in \eqref{eq:teleport-2} below), whose outcome
has $\T_d$-distance in $\{0,1\}$ whenever possible. 
In other words, the scenario in \eqref{eq:teleport} corresponds to the two random walks being far in $\T_d$, but close  in $G$ (i.e., ${\rm dist}_G(\varphi(Z_{t^-}),\varphi(W_{t^-}))\le 1$).
   In scenario \eqref{eq:teleport}, the update results in (write $(Z_{t^-},W_{t^-})=(z,w)\in \T_d\times \T_d$)
	\begin{equation}\label{eq:teleport-2}
	(Z_t,W_t)=\begin{dcases}
		(z,z) &\text{with probability}\ 1/4\\
		\begin{cases}(w,w) &\text{if}\ \varphi(z)\neq\varphi(w)\\
			(w',w')&\text{if}\ \varphi(z)=\varphi(w)
		\end{cases} &\text{with probability}\ 1/4\\
		(z,z') &\text{with probability}\ 1/4\\
		\begin{dcases}(w,w') &\text{if}\ \varphi(z)\neq \varphi(w)\\
			(w',w) &\text{if}\ \varphi(z)=\varphi(w)
		\end{dcases} &\text{with probability}\ 1/4\ ,
	\end{dcases}
\end{equation} 
with $z',w' \in \T^d$ such that 
\begin{equation}
	zz', ww'\in E(\T_d)\ ,\qquad \varphi(z)\varphi(z')=\varphi(w)\varphi(w')=xy\in E\ .
\end{equation}
In particular, when either $\varphi(z)=\varphi(w)=x$ or $\varphi(z)=x, \varphi(w)=y$, the outcomes in \eqref{eq:teleport-2} project  via $\varphi:\T_d\to G$  to $	(x,x), (y,y), (x,y), (y,x)$, respectively.

By construction, provided that $Z_0=W_0=o$,  $(\varphi(Z_t),\varphi(W_t))_{t\ge 0}$ has the same law as ${\rm CRW}(G)$ started from $(x,x)$.
 Hence, we obtain
\begin{equation}\label{eq:l2-lower-bound-crw2}
	\P^{\scriptscriptstyle{\rm CRW}(G)}_{x,x}\left(X_t=Y_t\right)= \P^{\scriptscriptstyle{\rm CRW}_*(\T_d)}_{o,o}\left(\varphi(Z_t)=\varphi(W_t)\right)\ge \P^{\scriptscriptstyle{\rm CRW}_*(\T_d)}_{o,o}\left(Z_t=W_t\right)\ .
\end{equation}
Our aim is to show that the  probability on the right-hand side above is larger than or equal to the analogous one for  ${\rm CRW}(\T_d)$, in which the aforementioned  non-nearest-neighbor updates  are discarded.
\begin{lemma}\label{lem:lb}
Let $G=(V,E)$ be any $d$-regular graph, $d\ge3$. Then, 
	\begin{equation}\label{eq:l2-lower-bound-crw3}
		\P^{\scriptscriptstyle{\rm CRW}_*(\T_d)}_{o,o}\left(Z_t=W_t\right)\ge \P^{\scriptscriptstyle{\rm CRW}(\T_d)}_{o,o}\left(Z_t=W_t\right)\ ,\qquad t \ge 0\ .
	\end{equation}
\end{lemma}
\begin{proof}[Proof of Proposition \ref{prop:lb}] The desired claim follows by combining  \eqref{eq:l2-lower-bound-crw2} and \eqref{eq:l2-lower-bound-crw3}.
\end{proof}
\begin{proof}[Proof of Lemma \ref{lem:lb}]
	Let $\Delta=\Delta_{\T_d}:=\{(z,z):z\in \T_d\}$. Let $S_t$ and $Q_t$ be the transition
	kernels of ${\rm CRW}_*(\T_d)$ and ${\rm CRW}(\T_d)$, with generators ${\mathcal L}_S$
	and ${\mathcal L}_Q$, respectively. Set $f_s:=Q_s{\bf 1}_\Delta$.
	By Duhamel's formula, for all $t\ge0$,
	\begin{equation}
		S_t((o,o),\Delta)-Q_t((o,o),\Delta)
		=\int_0^t \sum_{z,w\in \T_d} S_{t-s}((o,o),(z,w))\,({\mathcal L}_S-{\mathcal L}_Q)f_s(z,w)\,{\rm d}s\ .
	\end{equation}
	Moreover, ${\mathcal L}_S f_s(z,w)={\mathcal L}_Q f_s(z,w)$ unless
	\begin{equation}
		(z,w)\in{\mathcal B}:=\{(z,w):{\rm dist}_{\T_d}(z,w)\ge2,\ {\rm dist}_G(\varphi(z),\varphi(w))\le1\},
	\end{equation}
	so the sum may be restricted to ${\mathcal B}$.
	
	Fix $(z,w)\in{\mathcal B}$ and let $k:={\rm dist}_{\T_d}(z,w)\ge2$. Let
	\begin{equation}
		{\mathcal E}(z,w):=\{xy\in E:\{\varphi(z),\varphi(w)\}\subset\{x,y\}\}\ ,
		\qquad m(z,w):=|{\mathcal E}(z,w)|\ ,
	\end{equation}
	(counting edge multiplicity). For each $xy\in{\mathcal E}(z,w)$, when the clock of $xy$
	rings, ${\rm CRW}_*(\T_d)$ applies the teleport kernel at $(z,w)$, while ${\rm CRW}(\T_d)$
	does not have this update. Hence
	\begin{equation}
		({\mathcal L}_S-{\mathcal L}_Q)f_s(z,w)
		=\sum_{xy\in{\mathcal E}(z,w)} \Bigl({\bf E}[f_s(Z^+,W^+)\mid (Z^-,W^-)=(z,w),xy]-f_s(z,w)\Bigr) \ ,
	\end{equation}
	where $(Z^+,W^+)$ denotes the post-teleport state, cf.\ \eqref{eq:teleport-2}.
	By  invariance of ${\rm CRW}(\T_d)$,
	\begin{equation}
		f_s(u,v)=Q_s{\bf 1}_\Delta(u,v)=R_s({\rm dist}_{\T_d}(u,v),0), \qquad u,v\in \T_d\ ,
	\end{equation}
	with $R_s$ being the slow-bond BD kernel (Definition \ref{def:BD-chain}). In the teleport update at $(z,w)$, for each $xy\in{\mathcal E}(z,w)$,
	the four outcomes have, respectively, $\T_d$-distance  $0,0,1,k$ with weight $1/4$ each; therefore,
	\begin{equation}
		{\bf E}[f_s(Z^+,W^+)\mid (Z^-,W^-)=(z,w),xy]
		=\frac12 R_s(0,0)+\frac14 R_s(1,0)+\frac14 R_s(k,0)\ .
	\end{equation}
	Since $f_s(z,w)=R_s(k,0)$, we get
	\begin{equation}
		({\mathcal L}_S-{\mathcal L}_Q)f_s(z,w)
		=m(z,w)\left(\frac12\bigl(R_s(0,0)-R_s(k,0)\bigr)+\frac14\bigl(R_s(1,0)-R_s(k,0)\bigr)\right)\ge0\ ,
	\end{equation}
	because $\N_0\ni i\mapsto R_s(i,0)$ is decreasing (see, e.g., \cite[Lemma 4.1]{ding_lubetzky_peres_total_variation_2010}) and $k\ge2$.
	
	Thus, every integrand in the Duhamel formula is nonnegative, hence
	$S_t((o,o),\Delta)\ge Q_t((o,o),\Delta)$ for all $t\ge0$, which is the claim.
\end{proof}

\subsection{Upper bound} First, remark that Aldous-Lanoue's inequality in \eqref{eq:AL} in combination with  the quasi-Ramanujan property of random $d$-regular graphs (see, e.g., \cite{friedman_proof_2008,bordenave_new_2020}), that is, for all fixed $d\ge 3$,
\begin{equation}\label{eq:quasi-Ramanujan}
	\Prob\bigg({\rm gap}\le \frac{d}{2} (1-\varrho)-\delta\bigg)\xrightarrow{n\to \infty}0\ ,\qquad \delta>0\ ,
\end{equation}
yields the following non-sharp upper bound on the mixing time: for all $\eps, \delta >0$, w.h.p., 
\begin{equation}
		t_{{\rm mix},2}(\eps)\le \left(\frac{2}{d\left(1-\varrho\right)}+\delta\right)\log n \ .	
\end{equation}

We refine our estimate as follows.
	Fix $\varepsilon > 0$. Hence, for every $a \in \R$ and $b> 0$, we get
	\begin{align}
			&\Prob
		\left( \E_{\delta_1}\!\left[\left\|\frac{\eta_{T(a+b)}}{\pi}-1\right\|_{2}^2\right] >\varepsilon \right)\\
		&\qquad\le \Prob
		\left( \E_{\delta_1}\!\left[\left\|\frac{\eta_{T(a)}}{\pi}-1\right\|_{2}^2\right] > e^{{\rm gap}\,{b \log \log n}}\varepsilon \right)\\
		&\qquad\le 	\Prob\left(\E_{\delta_1}\!\left[\left\|\frac{\eta_{T(a)}}{\pi}-1\right\|_{2}^2\right]>
		e^{\frac{d(1-\varrho)\,b \log \log n}{4}}\varepsilon
		\right) + \Prob\left({\rm gap}\le \frac{d}4(1-\varrho)\right)\\
		&\qquad =
			\Prob\left(\E_{\delta_1}\!\left[\left\|\frac{\eta_{T(a)}}{\pi}-1\right\|_{2}^2\right]>
		e^{\frac{d(1-\varrho)\,b \log \log n}{4}}\varepsilon
		\right) + o(1)\ ,
	\end{align}
where for the first step we used Aldous-Lanoue's inequality in \eqref{eq:AL}, whereas the last step used \eqref{eq:quasi-Ramanujan}. 
Hence, by introducing the following event 
\begin{equation}
	\cW_t:=\left\{ \text{the orbits}\  (X_s,Y_s)_{s\le t}\ \text{may be embedded in a $d$-regular tree} \right\}\ ,
\end{equation}
we have, $\Prob$-a.s., 
\begin{align}\label{eq:L2-xxx}
	\begin{aligned}
	\P^{\scriptscriptstyle{\rm CRW}(G)}_{1,1}(X_t=Y_t) &\le \P^{\scriptscriptstyle{\rm CRW}(G)}_{1,1}(X_t=Y_t\,, \cW_t ) + \P^{\scriptscriptstyle{\rm CRW}(G)}_{1,1}(\cW_t^c)\\
	&\le \P^{\scriptscriptstyle{\rm CRW}(\mathds T_d)}_{o,o}(X_t=Y_t) + \P^{\scriptscriptstyle{\rm CRW}(G)}_{1,1}(\cW_t^c)\ ,
	\end{aligned}
\end{align}
where we recall that $\mathds T_d$ denotes the $d$-regular infinite tree and $o$ is an arbitrary vertex, which we will refer to as the \emph{root}.
By combining \eqref{eq:l2-norm} and \eqref{eq:L2-xxx}, Markov inequality yields
\begin{align}
	&\Prob\left(\E_{\delta_1}\!\left[\left\|\frac{\eta_{T(a)}}{\pi}-1\right\|_{2}^2\right]>
	e^{\frac{d(1-\varrho)\,b \log \log n}{4}}\varepsilon
	\right) \\
	&\qquad \le \frac1\eps\left\{
	\frac{n\,\P^{\scriptscriptstyle{\rm CRW}(\mathds T_d)}_{o,o}(X_{T(a)}=Y_{T(a)})-1}{\left(\log n\right)^{d(1-\varrho)\, b/4}} + \frac{n\,\Ex\left[\P^{\scriptscriptstyle{\rm CRW}(G)}_{1,1}({\cW_{T(a)}^c})\right]}{\left(\log n\right)^{d(1-\varrho)\, b/4}}\right\}
	\ .
\label{eq:markov-l2}
\end{align}

\begin{proof}[Proof of Theorem \ref{th:L2} (upper bound)]
	We prove  \eqref{eq:theorem-ub}  as soon as we show that, for some fixed $b>0$ and all $\eps>0$, the expression in curly brackets above vanishes by first taking $n\to \infty$, and then $a\to \infty$. This is indeed the case: the first term in \eqref{eq:markov-l2} vanishes for all $b>0$ by 
	\eqref{eq:CRW-infinite-tree-ub}; the second expression vanishes for $b>16/d(1-\varrho)$ in view of Lemma \ref{lemma:annealing} below.
\end{proof}

\begin{lemma}\label{lemma:annealing}
	Fix $d\ge 3$. Then, for all $n$ large enough and $t\le \frac{1}{2d}\log^2 n$, we have
	\begin{equation}
		\Ex\big[\P^{\scriptscriptstyle{\rm CRW}(G)}_{1,1}(\cW_{t}^c)\big]\le \frac{C\log^4 n}{n}\ ,
	\end{equation}
	for some $C=C(d)>0$.
\end{lemma}
\begin{proof}
	Let $\bar G$ be the $d$-regular configuration model on $|V|=n$ vertices, with $\bar \Prob$ and $\bar \Ex$ denoting law and expectation, and let
	$\mathscr S:=\{\bar G \text{ is simple}\}$. It is well-known that $\bar \Prob(\mathscr S)\ge \exp(-d^2)$, for all $n$ large enough (see, e.g., \cite[Eq.\ (2.1)]{lubetzky2010cutoff}). Moreover,  $G \overset{{\rm d}}= \bar G \mid \mathscr S$, and
	\begin{equation}
	\Ex\big[\P^{\scriptscriptstyle{\rm CRW}(G)}_{1,1}(\cW_t^c)\big]
	= \bar{\Ex}\big[\P^{\scriptscriptstyle{\rm CRW}(\bar G)}_{1,1}(\cW_t^c)\mid \mathscr S\big]	\le {\bar\Prob^{\scriptscriptstyle\rm ACRW}_{1,1}(\cW^c_t)}\,/\,{\bar \Prob(\mathscr S)}\ ,
	\end{equation}
	where, on the right-hand side,  the numerator is an \textit{annealed} probability. More generally, for any $x,y\in V$, the law $\bar\Prob_{x,y}^{\scriptscriptstyle\rm ACRW}$ is  defined through the following randomized algorithm:
	\begin{itemize}
		\item[(1)] Start with $|V|=n$ vertices, each equipped with $d\ge 3$ unmatched stubs, and two (labeled) particles whose positions are initially set as $X_0=x$ and $Y_0=y$.
		\item[(2)] At any time $s\ge 0$ in which any of the two particles sits on a vertex that has not yet matched all of its stubs, sequentially match its unmatched ones with others, uniformly at random. Let $\bar G_s$ denote the resulting graph, whose vertices are those that have at least one explored edge.
		\item[(3)] Let the two particles evolve as a ${\rm CRW}$-process  on the explored graph,  let $X_s, Y_s\in V$ denote their positions at time $s\ge 0$ and, whenever a particle jump occurs and a particle lands on a vertex with some unmatched stub, proceed as in step (2) above.
	\end{itemize}
This algorithm produces as an output a triple $(X_s,Y_s,\bar G_s)_{s\in[0,t]}$,  and the event $\cW_t^c$ implies that $\bar G_t$ is not a tree. In order to obtain some $\bar G_t$ that is not a tree, there must exist an $s\le t$ such that a new edge is created with a vertex  in $\bar G_{s^-}$. We refer to this event as a {collision}. 

Suppose that the cumulative number of jumps within time $t$ is smaller than some $A=A_n=o(n)$. Under this event, the probability that a given new matching results in a collision is bounded by $(d-1)A/(dn-2dA)$. Hence, 
\begin{equation}\label{eq:annealing}
\bar	\Prob_{1,1}^{\scriptscriptstyle\rm ACRW}(\cW^c_t)\le \frac{d A^2}{dn-2dA}+	\Prob({\rm Poi}(2dt)>A)\ ,
\end{equation}
By taking  $A>\log^2 n> 2dt$ and using the bound
\begin{equation}
\bar	\Prob({\rm Poi}(2dt)>A)\le \exp\left(-\frac{(A-2dt)^2}{2A} \right)\ ,
\end{equation}
we deduce that if, e.g., $A=2\log^2n$, the second probability  on the right-hand side in \eqref{eq:annealing} is $o(n^{-1})$. In conclusion, for all $n$ large enough, we have
\begin{equation}
\bar	\Prob_{1,1}^{\scriptscriptstyle\rm ACRW}(\cW^c_t)\le \frac{5 \log^4(n)}{n}\ ,
\end{equation}
from which the desired result follows.
\end{proof}

\section{Proof of Proposition \ref{prop:CRW-infinite-tree-CT}}\label{sec:CRW-infinite-tree}
Recall Definition \ref{def:BD-chain}, which introduces $R_t$,  the continuous-time slow-bond BD chain.
Consider the corresponding discrete-time version obtained from $R_t$ by taking $1/d$ as a time unit. Let $c_k$, $k\in \N_0$, be the probability of the event ``the discrete-time random walk started from $0$ is in $0$ at the $k$th-step'' (cf.\ \eqref{eq:ck-H} below). Clearly, $c_0=1$.

 Our main goal of the section is to prove the following asymptotics for $c_k$. Remark that, all throughout the section and unless stated otherwise, the asymptotic notation  is taken in the limit $k\to \infty$.
\begin{proposition}\label{prop:CRW-infinite-tree}
	Fix $d\ge 3$, and recall the definitions of $\varrho=\varrho_d$, $\sigma=\sigma_d$, and $\gamma =\gamma_d$ from \eqref{eq:gamma2}--\eqref{eq:gamma}.
	Then,  there exists $\beta=\beta_d\in \R$ (explicitly given in \eqref{eq:beta} below) satisfying
	\begin{equation}\label{eq:final-formula-ck}
		c_k \asymp k^\beta \gamma^k\ ,\qquad \text{as}\ k\to \infty\ .
	\end{equation}
\end{proposition}
We prove this result in Section \ref{sec:CRW-infinite-tree-proof} below.
In Section \ref{suse:discr-to-cont}, we translate the result in Proposition \ref{prop:CRW-infinite-tree} into a bound for the corresponding continuous-time process, and conclude therein the proof of Proposition \ref{prop:CRW-infinite-tree-CT}. 
\subsection{Preliminaries}\label{sec:BD-prel}We start by recalling some general facts.

Fix $p\in (1/2,1]$,  $p_0\in [0,1]$ as well as $q:=1-p$ and $q_1 \in [0,1-p]$, and consider a discrete-time BD chain on $\N_0$ with the following one-step transition probabilities:
\begin{equation}
	H(i,i+1)= p\ ,\qquad H(i+1,i) = q\ ,\qquad \text{for all}\  i \ge 1\ ,
\end{equation}
\begin{equation}
	H(0,1)=1-H(0,0)=p_0\ ,\qquad H(1,0)=q_1\ ,\qquad H(1,1)= 1-q_1-p
\end{equation}
For all $k \in \N_0$, let $(H^k(i,j))_{i,j \in \N_0}$ denote the $k$-step transition probabilities for this Markov chain.  
We extract precise asymptotics,  as $k\to \infty$, for $H^k(0,0)$  from the associated generating function $G_H(0,0|z)$, which is given by
\begin{equation}
	G_H(0,0|z):= \sum_{k\in \N_0} H^k(0,0)\, z^k\ ,\qquad z\in \mathbb C\ ,\ |z|\le\mathfrak r_H\ ,
\end{equation}
where $\mathfrak{r}_H$ is  the  radius of convergence of the series  $G_H(0,0|z)$ (we have $\mathfrak{r}_H\ge 1$).
We know that (see, e.g., \cite[Theorem 1.38]{woess_denumerable_2009}) 
\begin{equation}\label{eq:GH}
	G_H(0,0|z)=\frac{1}{1-U_H(0,0|z)}\ ,\qquad |z|<\mathfrak r_H\ ,
\end{equation}
where:
\begin{enumerate}
	\item $U_H(0,0|z)$ is given 	 as
	\begin{align}\label{eq:UH}
		U_H(0,0|z)&=z\, H(0,0) + z\, H(0,1)\, F_H(1,0|z)\\
		&= z\left(1-p_0\right)+zp_0\, F_H(1,0|z)\ ,
	\end{align}
	where, for all $|z|<\mathfrak{r}_H$,
	\begin{equation}\label{eq:FH}
		F_H(i,j|z):=\sum_{k\in \N_0}\P^H_i(\tau_j=k)\, z^k\ ,
	\end{equation}
	and $\tau_j$ is the first hitting time of $j\in \N_0$;	
	\item  $F_H(1,0|z)$ is given in terms of $F_H(2,1|z)$ (cf.\ \cite[Eq.\ (5.12)]{woess_denumerable_2009}):
	\begin{align}
		F_H(1,0|z)&=\frac{H(1,0)\,z}{1-H(1,1)\, z - H(1,2)\, z\, F_H(2,1|z)}\\
		&= \frac{q_1 z}{1-\left(1-p-q_1\right)z-pz\, F_H(2,1|z)}\  ,
	\end{align}
with  (see, e.g., \cite[Example 5.24]{woess_denumerable_2009})
\begin{equation}
F_H(2,1|z)=\frac{1}{2pz}\left(1-\sqrt{1-4 p q z^2}\right)\ ,
\end{equation}
which does not depend on $p_0, q_1$, but only on $p=1-q$.
\end{enumerate}

	\subsection{Proof of Proposition \ref{prop:CRW-infinite-tree}}\label{sec:CRW-infinite-tree-proof} Fix an integer $d\ge3$, and 
	set
	\begin{equation}\label{eq:p-q-p0-q1}
		p=1-q=\frac{d-1}{d}\in \bigg[\frac23\,,1\bigg)\ ,\qquad p_0=\frac12\ ,\qquad q_1=\frac{q}2\ .
	\end{equation}
	With this choice, the discrete-time chain $H^k$ from Section \ref{sec:BD-prel} coincides with the one introduced at the beginning of Section \ref{sec:CRW-infinite-tree}. In particular, we have
	\begin{equation}\label{eq:ck-H}
		c_k = H^k(0,0)\ ,\qquad k \in \N_0\ .
	\end{equation}	
	By specializing the general formulas from Section \ref{sec:BD-prel} to the setting in \eqref{eq:p-q-p0-q1}, we obtain that the Laurent series
	\begin{equation}\label{eq:laurent}
		\sum_{k\in \N_0}c_k	\, z^k\ ,\qquad z \in \mathds C\ ,
	\end{equation}
	reads, for $z\in \mathds C$ in a suitable neighborhood of the origin, as
	\begin{align}
	\label{eq:Nz}\begin{aligned}
		g(z):= G(0,0|z)
			%&=\frac{4-3\left(2p+1\right)z+4 p z^2+z  \sqrt{1-4pqz^2}}{\left(1-z\right)\left(4-2pz\left(3-z\right)\right)}\\
			&=\frac{ 4-3\left(2p+1\right)z+4 p z^2+z  \sqrt{1-4pqz^2}}{2p\left(1-z\right)\left(z_1-z\right)\left(z_2-z\right)}\ ,
		\end{aligned}
	\end{align}
	where
	\begin{equation}\label{eq:z1-z2}
		z_1= \frac{3}{2}-\sqrt{\frac{{9p-8}}{ 4p}}\ ,\qquad z_2=\frac{3}{2}+\sqrt{\frac{{9p-8}}{4p}}\ .
	\end{equation}
	Note that $z_1, z_2$ are complex conjugate if and only if $p<8/9$ (i.e., $d<9$). Moreover (cf.\ \eqref{eq:gamma} and \eqref{eq:z1-z2}), 
	\begin{equation}\label{eq:sigma-z2}
		z_2=\sigma^{-1}\ ,\qquad d \ge 10\ .
	\end{equation}

	In what follows, we determine  asymptotics of the coefficients $(c_k)_{k\in \N_0}$ of the  series in \eqref{eq:laurent} via a singularity analysis of 
  the function $g(z)$ in \eqref{eq:Nz}.  
	Start by noting that:
	\begin{enumerate}
		\item the numerator in \eqref{eq:Nz} is holomorphic on $|z|<1/\varrho$, where \begin{equation}\label{eq:rho-p-q}
			\varrho=2\sqrt{pq}=\frac{2\sqrt{d-1}}d\ ,
			\end{equation} being sum of holomorphic functions on that same open disk;
		\item 		
		there are five ``potential singularities'' at	
		\begin{equation}
			1\ ,\ z_1\ ,\ z_2\ ,\ \varrho^{-1}\ ,\ -\varrho^{-1}\ .
		\end{equation}
	\end{enumerate}   We now characterize the radius of convergence of the series in \eqref{eq:laurent}.
	\begin{figure}
		\includegraphics[width=7cm]{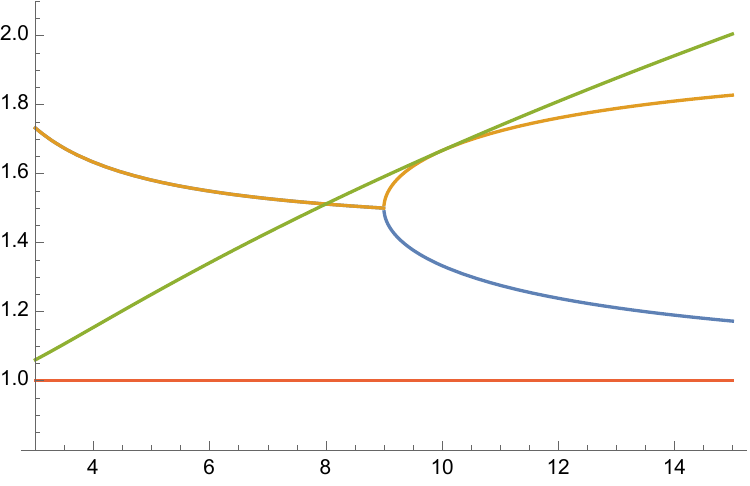}
		\caption{The horizontal axis corresponds to the variable $d\ge 3$. The function in red is constantly equal to $1$. The orange curve is the function $d\mapsto |z_2|$,  while the blue one is $d\mapsto |z_1|$ (note that the two curves coincide for $d\le 9$). In green,  the curve $\varrho^{-1}=d/(2\sqrt{d-1})$.}\label{fig:pot-sing}
	\end{figure}

	\begin{lemma}\label{lemma:rd}
		Let $r=r_d\ge 0$ be the distance from the origin of the nearest singularity of the function $g(z)$ (see \eqref{eq:Nz}) when $p=\frac{d-1}{d}$. Then, 	
		\begin{equation}
		r=\begin{cases}
				\varrho^{-1}&\text{if}\ d\le 10\\
				z_2&\text{if}\ d\ge 10\ .
			\end{cases}
		\end{equation}
	\end{lemma}
	\begin{proof}
		We start by proving that, for all $d\ge 3$, $z=1$ is not a singularity. For this purpose, it suffices to rewrite the function $g(z)$ given in \eqref{eq:Nz} as follows
		\begin{equation}
			g(z)=\frac{(1-z)\left[3(2p+1)+4p(1-z)-\sqrt{1-4pqz^2}\right]+\left\{1-10p+8pz+\sqrt{1-4pqz^2} \right\}}{2p(1-z)(z_1-z)(z_2-z)}\ ,
		\end{equation}
	observe that the numerator above is holomorphic at $z=1$ and that, for $p\ge\frac12$, the quantity between curly brackets in the previous display equals zero at $z=1$. This proves the claim that $z=1$ is not a singularity for $g(z)$. 
	
	Now we 	observe that  (see also Figure \ref{fig:pot-sing})
\begin{align}\label{eq:u-v2}
		\begin{split}1/\varrho\le |z_2|\end{split} \begin{split}\begin{array}{c}
				\text{if and only if}\ d\in \{3,\ldots,8\} \cup \{10\}\ ,\\
				\text{with equality only for $d=8,10$}\ .
			\end{array}
		\end{split}
	\end{align}
Moreover, for $d\ge 9$,
	\begin{equation}\label{eq:z1z2d>8}
		z_1, z_2 \in \R \quad  \text{and}\quad z_1\le z_2\ 	 \text{with equality if and only if}\ d=9\ ,
	\end{equation} 
while $z_1$ and $z_2$ are complex conjugate for $d\in\{3,\dots,8\}$. Finally, 	\begin{align}\label{eq:z1<1/rho-d>8}
	z_1<1/\varrho\ ,\qquad d\ge 9\ .
\end{align}
  Hence, we are only left to show that $z_1$ is not a singularity for the function $g(z)$ as soon as $d\ge 9$.  
Indeed, by expanding the square root around $z=z_1$ up to the second order (recall \eqref{eq:z1<1/rho-d>8}),  the numerator on the right-hand side of \eqref{eq:Nz} reads as
	\begin{align}
		&\left\{4-3\left(2p+1\right)z_1+4pz_1^2 + z_1\sqrt{1-4pqz_1^2} \right\}  
		\\
		&\qquad + \left\{3\left(2p+1\right)-8pz_1+\frac{4pqz_1}{\sqrt{1-4pqz_1^2}}-\sqrt{1-4pqz_1^2}\right\}\left(z_1-z\right)+O((z_1-z)^2)\ .
	\end{align}
Since the first  curly bracket above equals zero for all $d\ge 9$, while the second one equals zero for $d=9$, 
 $z=z_1$ is not a singularity for $g(z)$ when $d\ge 9$. This concludes the proof of the lemma.
	\end{proof}
By the \emph{first principle of coefficient asymptotics} (see, e.g., \cite[Theorem IV.7]{flajolet_analytic_2009}),  Lemma \ref{lemma:rd} is already enough to obtain the exponential decay rates prescribed in Proposition \ref{prop:CRW-infinite-tree} (cf.\ \eqref{eq:sigma-z2}--\eqref{eq:rho-p-q}). In order to determine the coefficients $\beta$ in \eqref{eq:final-formula-ck}, we use the standard machinery of singularity analysis; see, e.g., \cite[Chapter VI]{flajolet_analytic_2009}. 
 \begin{proof}[Proof of Proposition \ref{prop:CRW-infinite-tree}]In view of the previous lemma, we get the desired claim (i.e.,  \eqref{eq:ck-9}, \eqref{eq:ck-10} and \eqref{eq:ck-11}  below) by proving that the coefficients $\beta=\beta_d$ in \eqref{eq:final-formula-ck} satisfy
	\begin{equation}\label{eq:beta}
		\beta=\begin{cases}
			-\frac32&\text{if}\ d\le 9\\
			-\frac12&\text{if}\ d=10\\
			0&\text{if}\ d\ge 11\ 	.
		\end{cases}
	\end{equation}
	We  split the proof into three cases: $d\le 9$, $d=10$, and  $d\ge 11$.
	
	\smallskip \noindent
	\emph{Case $d \le 9$.} Lemma \ref{lemma:rd} ensures that $r=\varrho^{-1}$. Moreover, $\varrho^{-1}$ and $\varrho$ are two singularities and, by \eqref{eq:u-v2}, $z_1$ and $z_2$ are not singularities for all $d \in \{3,\ldots, 7\}\cup \{9\}$. We argue that $z_1$ and $z_2$ are not singularities also when $d=8$. Indeed, in this case, 
	\begin{align}
			z_1=\frac32-\mathbf{i} \frac{1}{2\sqrt{7}}\ ,\qquad z_2=\frac32+\mathbf{i} \frac{1}{2\sqrt{7}}\ ,
	\end{align} 
and, by expanding the numerator in \eqref{eq:Nz} around $z=z_j$, $j=1,2$, the zeroth-order term equals
\begin{equation}
	4-3\left(2p+1\right)z_j+4pz_j^2 + z_j\sqrt{1-4pqz_j^2} =0\ ,\qquad j=1,2\ .
\end{equation}
In conclusion, for $d\le 9$, $\pm\varrho^{-1}$  are the only two singularities at distance $r$ from the origin. 

Let us now expand the function $g(z)$ around both singularities:
\begin{align}
	g(z)&= a_+ -f_+\sqrt{1-\varrho z} + O\left(1-\varrho z\right)\ ,&& \text{as}\ z\to \varrho^{-1}\ ,\\
	g(z)&= a_- -f_-\sqrt{1+\varrho z} + O\left(1+\varrho z\right)\ ,&&\text{as}\ z\to -\varrho^{-1}\ ,
\end{align}
for some constants $a_+, a_-\in \R$ and values
\begin{align}
	f_+&:=- \frac{\varrho^{-1}}{\sqrt 2p\left(1-\varrho^{-1}\right)\left(z_1-\varrho^{-1}\right)\left(z_2-\varrho^{-1}\right)}>0\ ,
	\\
	f_-&:=- \frac{-\varrho^{-1}}{\sqrt 2p\left(1+\varrho^{-1}\right)\left(z_1+\varrho^{-1}\right)\left(z_2+\varrho^{-1}\right)}>0\ .	
\end{align}
Since $f_+>f_->0$, by \cite[Theorem VI.5]{flajolet_analytic_2009}, we conclude that
\begin{equation}\label{eq:ck-9}
	c_k\sim\frac{f_++(-1)^k f_-}{\sqrt{\pi}}\, k^{-3/2} \varrho^k\ ,\qquad  \text{as}\ k\to \infty\ .
\end{equation}

\smallskip \noindent
\emph{Case $d=10$.} In this case $r=z_2=\varrho^{-1}=\frac53$. There are two singularities at distance $r=\frac53$ from the origin, namely $\pm\frac53$. In particular,
\begin{align}
	g(z)&=-9+\frac{25}6{\sqrt{\frac65}}\left(\frac{5}{3}-z\right)^{-1/2}+O\left(\frac53-z\right)\ ,&& \text{as}\ z\to \frac53\ ,\\
	g(z)&=\frac7{12}-\frac{5\sqrt{2}}{144}\sqrt{1+\frac{3}{5}z}+O\left(\frac53+z\right)\ ,&& \text{as}\ z\to -\frac53\ 	.
\end{align}
Hence, by \cite[Theorem VI.5]{flajolet_analytic_2009}, we conclude that the singularity at $-5/3$ gives a lower order contribution to the the coefficients of $g$, thus,
\begin{equation}\label{eq:ck-10}
	c_k\sim 5\,\sqrt{\frac5{6\pi}}\,  k^{-1/2} \varrho^k\ ,\qquad \text{as}\  k\to \infty\ .
\end{equation}

\smallskip \noindent
\emph{Case $d\ge 11$.} By Lemma \ref{lemma:rd}, we have $r=z_2$. Moreover, $z_2<\varrho^{-1}$, and $z=z_2$ is the only singularity at distance $r$ from the origin. By expanding $g$ around $z=z_2$, one finds explicit constants $\mathfrak h =\mathfrak h_d\in(0.7,2)$ satisfying
\begin{align}
	g(z)= \frac{\mathfrak h}{z_2-z}+O\left(1\right)\ ,\qquad \text{as}\ z\to z_2\ .
\end{align}
By \cite[Theorem VI.4]{flajolet_analytic_2009}, the coefficients in \eqref{eq:laurent} satisfy
\begin{align}\label{eq:ck-11}
	c_k \sim \mathfrak h z_2^{-k}\ ,\qquad \text{as}\ k\to \infty\ .
\end{align}

This  concludes the proof.
\end{proof}

\subsection{From discrete- to continuous-time: proof of Proposition \ref{prop:CRW-infinite-tree-CT}}\label{suse:discr-to-cont}
Recall  that $R_t$ denotes the transition kernel of the continuous-time biased random walk on $\N_0$, with a bias to the right given by $p=\frac{d-1}{d}$ and with $01$ as a slow bond, cf.\ \eqref{eq:rates-slow-bond}. 
Hence, 
\begin{align}\label{eq:R-c}
	R_t(0,0)= \sum_{k=0}^\infty e^{-dt}\, \frac{(dt)^k}{k!}\,c_k\ ,\qquad t \ge 0\ ,
\end{align}
where $(c_k)_{k\in \N_0}$ are the discrete-time transition probabilities as introduced at the beginning of Section \ref{sec:CRW-infinite-tree}. Because of this identity,  Proposition \ref{prop:CRW-infinite-tree}, and \eqref{eq:P-CRW-R}, we now have all the ingredients to prove Proposition \ref{prop:CRW-infinite-tree-CT}.
\begin{proof}[Proof of Proposition \ref{prop:CRW-infinite-tree-CT}]
	By Proposition \ref{prop:CRW-infinite-tree}, which ensures that $c_k\asymp k^\beta\gamma^k$,
	\eqref{eq:R-c} further reads as
	\begin{align}
		R_t(0,0)&\asymp 
		e^{-dt}+\sum_{k=1}^\infty e^{-dt}\,\frac{(d\gamma t)^k}{k!} \, k^\beta = e^{-dt}+ e^{-d\left(1-\gamma\right)t} \sum_{k=1}^\infty e^{-d\gamma t}\,\frac{(d\gamma t)^k}{k!}\, k^\beta\ ,\qquad t\ge 0\ .
	\end{align}
	Since $\beta$-moments, $\beta \le 0$, of a (positive) Poisson random variable with diverging rate $\lambda\to \infty$ are asymptotically equal to $C \lambda^\beta$, for some $C=C_\beta>0$ (see, e.g., \cite{jones_approximating_2004}), we get
	\begin{equation}\label{eq:R-t-infinity}
		R_t(0,0) \asymp e^{-d(1-\gamma)t}  \,	t^\beta\ ,\qquad \text{as}\ t\to \infty\ .
	\end{equation}
	After recalling from \eqref{eq:P-CRW-R} that $R_t(0,0)=\mathbf P_{o,o}^{\scriptscriptstyle{\rm CRW}(\T_d)}(X_t=Y_t)$,  both claims in \eqref{eq:CRW-infinite-tree-lb}--\eqref{eq:CRW-infinite-tree-ub} follow at once by	 inserting $t=T(a)$ as given in \eqref{eq:def-Ta} into \eqref{eq:R-t-infinity}.
\end{proof} 

We conclude with two remarks on the large-degree regime.
\begin{remark}[Sharpening of the cutoff window]\label{rem:window}
	In the supercritical regime ($d>10$),  we expect changes not only in the cutoff location but also in the cutoff window: since $\beta=0$, the window should shrink from order $t_{\rm rel}\log\log n$ to order $t_{\rm rel}$. A rigorous proof of this claim would need an improved version of the upper bound in \eqref{eq:markov-l2}.
\end{remark}
\begin{remark}[Regime $d\to \infty$]\label{rem:large}
	Since the constant $\mathfrak h=\mathfrak h_d $ appearing in \eqref{eq:ck-11} remains bounded away from zero and infinity,	 the asymptotic equivalence  $c_k\asymp \sigma^{k}$ (and, thus, $R_t(0,0)\asymp e^{-d(1-\sigma)t}$ in \eqref{eq:R-t-infinity}) holds uniformly over all $d\ge 11$.   
\end{remark}

\subsection*{Acknowledgments}
The authors would like to thank all organizers of the CIRM conference \textquotedblleft{Interacting particle systems and related fields}\textquotedblright, and CIRM for the kind hospitality.
While this work was written, M.Q.\ and F.S.\ were associated to INdAM (Istituto Nazionale di Alta
Matematica “Francesco Severi”) and the group GNAMPA. The same authors acknowledge partial support from the GNAMPA-INdAM project “Stochastic exchange models: from kinetic theory to opinion dynamics”.
 F.S.\ acknowledges financial support by the UniMI Early Career Development 2025.

\subsection*{Data availability} Data sharing not applicable to this article as no datasets were generated
or analyzed during the current study.

\subsection*{Conflicts of interest} All authors declare that they have no conflicts of interest.

\bibliographystyle{alpha}

	\end{document}